\documentclass[12pt, reqno]{amsart}
\usepackage{amsmath, amsthm, amscd, amsfonts, amssymb, graphicx, color}
\usepackage[bookmarksnumbered, colorlinks, plainpages]{hyperref}
\usepackage{graphicx}
\usepackage{caption}
\usepackage{subcaption}
\makeatletter \oddsidemargin.9375in \evensidemargin \oddsidemargin
\def\authorsaddresses#1{\dedicatory{#1}}

\def\authorsaddresses#1{\dedicatory{#1}}

\numberwithin{equation}{section}
\newcommand{\A}{\mathcal{A}}

\newtheorem{theorem}{Theorem}[section]
\newtheorem{lemma}[theorem]{Lemma}

\newtheorem{corollary}[theorem]{Corollary}
\theoremstyle{definition}
\newtheorem{definition}[theorem]{Definition}
\newtheorem{example}[theorem]{Example}

\theoremstyle{remark}

\numberwithin{equation}{section}

\begin{document}
\setcounter{page}{1}



\title[Homology for one-dimensional solenoids]{Homology for one-dimensional solenoids}

\author[Massoud Amini, Ian F. Putnam and Sarah Saeidi Gholikandi]{ Massoud Amini$^1$, Ian F. Putnam$^2$
and  Sarah Saeidi Gholikandi$^1$}

\authorsaddresses{$^1$ Department of  Mathematics, University
of Tarbiat Modares, P. O. Box 14115-111, Tehran, Iran.\\
Sarahsaeadi@gmail.com, mamini@modares.ac.ir\\
\vspace{0.5cm} $^2$ Department of Mathematics and Statistics\\
University of Victoria,
Victoria, B.C., Canada V8W 3R4\\
ifputnam@uvic.ca}
\subjclass[2010]{Primary 55N35; Secondary 37D99, 37B10.}

\keywords{Smale spaces, one-dimensional generalized  solenoids, homology.}

\begin{abstract}
Smale space is a particular class of
 hyperbolic topological dynamical systems, defined
  by David Ruelle. The third author constructed a
  homology theory for Smale spaces
which is based on Krieger's dimension group invariant
for shifts of finite type. In this paper, we compute
this homology for the one-dimensional generalized
solenoids of R.F. Williams.

\end{abstract}
\maketitle

\section*{Introduction} Smale spaces were  defined by
David Ruelle as a purely topological version of the
basic sets of Axiom A systems which arise in Smale's
program for differentiable dynamics \cite{D1,S1,AH,F1,F2}.
Informally, a pair $(X,\varphi)$, where $X$ is
a compact metric space and $\varphi$ a homeomorphism
of $X$, is a Smale space if it possesses local
coordinates in contracting and expanding directions.
 The precise definition  will be given in Definition
  \ref{1-2}.  Hyperbolic toral automorphisms,
  one-dimensional generalized solenoids as
  described by R.F. Williams
   and shifts of finite type are all examples
   of Smale spaces. In fact, shifts of
finite type play a particularly important role
in the subject. In particular, Rufus Bowen \cite{B1}
proved that
every irreducible Smale space is the image of a shift of
finite type under a finite-to-one factor map. In
the late 1970's, W. Krieger  \cite{K1} introduced
a pair of invariants for  a shift of
 finite type, called the past
and future  dimension groups.
Building on these two ideas, the third author \cite{P1}
has shown
the existence of a homology theory for Smale spaces
whose existence had been conjectured earlier by
Bowen \cite{B2} and which generalizes Krieger's invariant.

The main goal of this paper is to compute the
homology of one-dimensional generalized solenoids
defined  by Williams \cite{W1,W2},  generalized by Inhyeop Yi
\cite{Y} and later by Klaus   Thomsen  \cite{T1}.
 In fact  these spaces are inverse limits of
  finite graphs with one expanding map $f$. This is a
natural first class to consider for the homology
 beyond shifts of finite type, where the theory simply
 re-captures Krieger's invariants. First of all, the
 shifts of finite type which cover the solenoids are
 particularly simple and can be written quite explicitly.
 Secondly, since the solenoids have totally disconnected
 stable sets, the factor map can be chosen to have
  the property that
 it is $s$-bijective (see section 1). In this case, the
 computation of the homology is simplified significantly.

The paper is organized as follows. In  the first section,
 we briefly review Smale spaces and shifts of finite type.
The second section summarizes results on
one-dimensional solenoids. Much of this is a summary
of the work of Yi and Thomsen, but we also establish
some new results which will be needed  in our
computations later. In the final section, we begin by stating
our main results. The
remainder of that section is occupied with the calculation
of the homology.

\section{Smale spaces}

\begin{definition}
\label{1-2} Suppose that $(X,\varphi)$  is a compact metric space
and $\varphi$ is a homeomorphism of  $X$. Then $(X,\varphi)$
is called a Smale space if  there exist
constants  $\varepsilon_{X}$ and $0 < \lambda < 1$ and a
continuous map from $$\triangle_{\varepsilon_{X}}=\{(x,y) \in X\times X \quad| \quad d(x,y) \leq \varepsilon_{X}\quad \}$$
to  $X$ (denoted with $[,]$) such that:\\

$\emph{B} \ 1   \quad  [x,x]=x, $

$\emph{B}  \ 2  \quad  [x,[y,z]]=[x,z],$

$\emph{B}  \  3    \quad    [x,y],z]=[x,z] ,$

$\emph{B}  \ 4  \quad    [\varphi(x),\varphi(y)] =[x,y] $,

$\emph{C} \ 1   \quad    d(\varphi(x),\varphi(y))\leq \lambda  \, d(x,y), \text{ whenever } [x,y]=y, $

$\emph{C}  \ 2   \quad    d({\varphi}^{-1}(x),{\varphi}^{-1}(y) \leq \lambda \, d(x,y), \text{ whenever } [x,y]=x, $
whenever both sides of an equation are defined.
\end{definition}

Examples of Smale spaces include  solenoids,
substitution tiling spaces, the basic sets
for Smale's Axiom A systems and shifts of finite type.

Let $(X,\varphi)$ be a Smale space. For any $x$ in $X$ and  $0<\varepsilon \leq \varepsilon_{X}$, we define
$$X^{s}(x,\varepsilon)=\{ y \ | \quad d(x,y)\leq\varepsilon \ ,[x,y]=y\}$$
$$X^{u}(x,\varepsilon)=\{ y \  | \quad d(x,y)\leq\varepsilon \ ,[x,y]=x\}$$
These sets are called local stable and local unstable sets.

 We say that two points $x$ and $y$ in $X$ are stably (or unstably) equivalent if
$$ \lim_{n\rightarrow +\infty}d(\varphi^{n}(x),\varphi^{n}(y))=0  \hspace{1cm}
(or \lim_{n\rightarrow -\infty}d(\varphi^{n}(x),\varphi^{n}(y))=0, \text{resp.}).$$
Let $X^{s}(x)$  and  $X^{u}(x)$  denote the stable
and unstable equivalence classes  of $x$, respectively.
As the notation would suggest, there is a close connection
between local stable sets and  stable equivalence classes
(see Chapter 2 of \cite{P1}).

We recall that a factor map between two Smale
spaces $(Y, \psi)$ and $(X, \varphi)$ is
a continuous function $\pi:Y \rightarrow X$ such that $\pi \circ \psi = \varphi \circ \pi$. Of particular importance
in the homology theory are factor maps which are
$s$-bijective: that is, for each $y$ in $Y$, the restriction
of $\pi$ to $Y^{s}(y)$ is a bijection to $X^{s}(\pi(y))$.
There is  obviously an analogous definition of a
$u$-bijective factor map, which will not be needed here.

If $(X, \varphi)$ is a Smale space, then so
is $(X, \varphi^{n})$, for every positive integer
$n$. In fact, these two Smale spaces
have exactly the same stable and unstable equivalence
relations. Somewhat more subtlely, they have
naturally isomorphic homology theories in the
sense of \cite{P1}. If one takes the view that the
homology theories produce a sequence of abelian groups
together with canonical automorphisms
induced by $\varphi$ (see Chapter 3 of \cite{P1}), then,
while groups are the same, the automorphism of the
latter is simply the  $n$th power of that
of the former. As our attention will be mainly
in computing the groups themselves, we will be
quite happy to replace $\varphi$ by $\varphi^{n}$.

\subsection{Shifts of finite type}

 A graph $G$ consists of finite sets $G^{0}$ and $G^{1}$ and maps  $i,t:G^{1}\rightarrow G^ {0}$.
The elements of $G^{0}$ are called vertices and the elements of $G^{1}$ are called edges. The notation for the maps is meant to suggest initial and terminal and the graph is drawn by depicting each vertex as a dot and each edge e as an arrow from $i(e)$ to
$t(e)$.
To any graph $G$, we associate the following
dynamical system:
$$\Sigma_{G}=\{ \ {(e_n)}_{n \in \Bbb Z} \  |\ e_{n} \in G^{1},  \  t(e_{n})=i(e_{n+1}) \  for \  all \   n \in \Bbb Z \},$$
$${(\sigma(e))}_{n}={e}_{n+1}.$$
For any $e$ in $\Sigma_{G}$ and $K \leq L$, we let
$e_{[K,L]} = (e_{K}, e_{K+1}, \ldots, e_{L})$. It is also
convenient to define $e_{[K+1,K]} = t(e_{K})=i(e_{K+1})$.
We use the metric
$$
d(e,f) = \inf\{ 1, 2^{-K-1} \mid K \geq 0, e_{[1-K,K]}
= f_{[1-K,K]} \}
$$ on $\Sigma_{G}$.

It is then easy to see that $(\Sigma_{G}, \sigma)$
is a Smale space with constants $\varepsilon_{X}=\lambda=\frac{1}{2}$ and
\begin{eqnarray*}
{[e,f]}_{ k}=\left  \{ \begin{array}{ll}
f_{k} & \  k\leq 0 \\ e_{k} &  \ k \geq 1. \end {array} \right.
\end{eqnarray*}

While the precise definition of a shift of finite type
is slightly different (see \cite{L1}), every shift of
finite type is conjugate to $(\Sigma_{G}, \sigma)$, for some graph $G$.

 Let $G$ be a graph and let $K \geq 2$. A path of length $K$
in $G$ is a sequence $(e_1,e_2,...,e_K)$ where $e_k$ is in $G^1$, for each $1 \leq k \leq K$
and $t(e_k) = i(e_{k+1})$, for $1 \leq k < K$ .
 We  let $G^K$ denote the set of all paths of length $K$
 in $G$ and, simultaneously, the graph whose vertex set is $G^{K-1}$
and whose edge set is $G^K$ with initial and terminal maps $$i(e_1,e_2,...,e_K)=(e_1,e_2,...,e_{K-1}), \quad t(e_1,e_2,...,e_K)=(e_2,e_3,...,e_{K}).$$

We recall the computation of the invariants
$D^{s}(\Sigma_{G}, \sigma)$ and $D^{u}(\Sigma_{G}, \sigma)$. We let $\mathbb{Z} G^{K}$ denote the
free abelian group of the set $G^{K}$,
for any $K \geq 0$. If $A$ is some subset
of $G^{K}$, we let $Sum(A) = \Sigma_{a \in A} a$.
  The initial
and terminal maps $i, t: G^{K} \rightarrow G^{K-1}$
induce group homomorphisms, also denoted $i,t$
from $\mathbb{Z} G^{K}$ to $\mathbb{Z} G^{K-1}$.
In addition, we have a group homomorphism
$t^{*}: \mathbb{Z} G^{K-1} \rightarrow \mathbb{Z} G^{K}$ defined by
$t^{*}(e) = Sum(t^{-1}\{ e \})$. We define the map
$\gamma_{G}^{s} = i \circ t^{*}  $ and
$D^{s}(G^{K})$ is defined to be
the inductive limit of the sequence
\[
\mathbb{Z} G^{K-1} \stackrel{\gamma_{G}^{s}}{\rightarrow}
\mathbb{Z} G^{K-1} \stackrel{\gamma_{G}^{s}}{\rightarrow}
\cdots
\]
As explained in \cite{P1}, the results for
different values of $K$ are all naturally
isomorphic. In fact, the map $i$ induces and
isomorphism from $D^{s}(G^{K})$ to $D^{s}(G^{K-1})$.
These groups are all isomorphic to
$D^{s}(\Sigma_{G}, \sigma)$.
There are analogous definitions of $i^{*}$,
$\gamma_{G}^{u} = t \circ i^{*}$ and
$D^{u}(G^{K})$.

\section{ One-dimensional generalized  solenoids }

The main  goal of this paper is to compute  the homology of one-dimensional  generalized  solenoids, defined by K. Thomsen, based on an earlier work of  Yi \cite{T1,Y}.  The solenoids were defined first by Williams on manifolds, and then generalized  by Yi on topological spaces \cite{W1,W2,Y}, but also
see \cite{W3}. We first quickly review the Thomsen's definition \cite{T1}.

\begin{definition}
\label{2-1}
Let $X$ be a finite (unoriented), connected
 graph with vertices $V$ and edges $E$.
 Consider a continuous map $f : X \rightarrow X $.
  We say that   $(X, f)$ is   a
pre-solenoid if the
following conditions are satisfied for
some metric $d$ giving the topology of $X$:\\
$\alpha$) (expansion) there are
constants $C > 0$ and $\lambda > 1$ such that
\newline
$d(f^n(x), f^n(y)) \geq  C \lambda ^nd(x, y)$
for every $n \in  \mathbb{N}$ when
 $x, y \in  e \in  E$ and there is
 an edge $e'\in  E $ with $f^n([x, y]) \in e'$
($[x, y]$ is the interval in $e$ between
$x$ and $y$),\\
$\beta$) (non-folding) $f^n$ is locally
injective on $e$ for each $e \in E$ and
each $n \in \mathbb{N}$,\\
$\gamma$) (Markov) $f(V) \subset V$,\\
$\delta$) (mixing) for every edge $e\in E$,
 there is  $m \in  \mathbb{N} $ such that
 $X \subseteq f^m(e)$,\\
$\epsilon$) (flattening) there is
 $d \in \mathbb{N}$ such that for
 all $x \in X$ there is a neighbourhood
 $U_x$ of $x$ with $f^d(U_x)$ homeomorphic
 to $(- 1, 1)$.
\end{definition}

We usually refer to a $d$ satisfying the flattening
condition as the \emph{flattening number} of $f$.

Suppose that $(X, f)$ is a pre-solenoid. Define $$\overline{X} =\{(x_i)_{i=0}^\infty  \in  X^{\Bbb N \cup \{0\}} : f (x_{i+1}) = x_i , i = 0, 1, 2, \cdots\}$$
Then  $\overline{X}$ is  a compact metric space with the metric  $$D ((x_i)_{i=0}^\infty,(y_i)_{i=0}^\infty)=\sum^{\infty}_{i=0} 2^{-i}d (x_i, y_i).$$
We also define $\overline{f} : \overline{X} \rightarrow \overline{X}$ by  $\overline{f}(x)_i = f (x_i)$
for all $i \in \Bbb N \cup \{0\}$. It
 is a homeomorphism with inverse
$$\overline{f}^{-1}(x_0, x_1, x_2, \cdots ) = (x_1, x_2, x_3, \cdots ).$$
Finally, we define the map $\pi: \overline{X} \rightarrow X$ by
\[
\pi(x_{0}, x_{1}, x_{2}, \ldots ) = x_{0},
(x_{0}, x_{1}, x_{2}, \ldots ) \in \overline{X}.
\]

Following Williams and Yi \cite{W2,Y},  Thomsen called $(\overline{X}, \overline{f})$   a generalized one-dimensional
solenoid or just a one-solenoid \cite{T1}.

\begin{definition}
\label{2-2}
Let $(X, f)$ be a pre-solenoid. The
system $(\overline{X}, \overline{f})$
is called a generalized one-solenoid.
\end{definition}

\begin{theorem}
\label{2-3}
\cite{T2} One-dimensional generalized solenoids are Smale spaces.\end{theorem}

It was shown by Williams that expanding attractors of certain diffeomorphisms
of compact manifolds are one-solenoids via a conjugacy which turns the restriction of the
diffeomorphism into $f$. He also showed that each one-solenoid arises in this way from a diffeomorphism of the 4-sphere \cite{T1}.

 Considering a pre-solenoid $(X, f)$,
 an orientation of $X$  is  defined  to be a collection of homeomorphisms $\psi_e: (0, 1) \rightarrow e$, $e \in E$. We say that $f$ is positively
(respectively,
negatively) oriented with respect to the orientation $\psi_e
, e \in  E$, when the function $$\psi_{e'}^{-1} o f o \psi_e : \psi_e^{-1}(e \cap f^{-1}(e')) \rightarrow  [0, 1] $$
is increasing (respectively,  decreasing) for each $e, e' \in  E$. A pre-solenoid $(X, f)$ is positively
(respectively, negatively) oriented when there is an orientation of the edges in $X$ such
that $f$ is positively  (respectively, negatively) oriented with respect to that orientation.
$(X, f)$ is oriented when it is either positively or negatively oriented. Notice that if
$(X, f)$ is oriented, then
$(X, f^{2})$ is positively oriented.
We say that $(X, f)$ is orientable if $X$ has
an orientation making $(X,f)$ oriented.

The 1-solenoid $(\overline{X}, \overline{f})$ is orientable when there is an oriented pre-solenoid $(X_1, f_1)$ such that $(\overline{X}, \overline{f})$
is conjugate to $(\overline{X_1}, \overline{f_1})$. When $(X_1, f_1)$ can be chosen to be positively (resp, negatively) oriented, we say that $(\overline{X}, \overline{f})$ is positively (resp,  negatively) orientable\cite{T1}.

\begin{theorem}
\label{2-4} \cite{T1} Let $(X, f)$ be a pre-solenoid. Then $(\overline{X}, \overline{f})$ is positively (resp, negatively) orientable  if and only if $(X,f)$ is positively (resp, negatively) oriented.  \end{theorem}

We will give a new proof of this fact which is based
on our computations of homology.

We observe that if $(X, f)$ is
 a pre-solenoid. Then for $ n  \in \Bbb N$,  $(X, f^n)$ is also a  pre-solenoid.
Moreover, if
 $d$ is a flattening number, then $(X, f^{n})$,
 $n \geq d $ is  a pre-solenoid whose flatting
 number is one.
 Also observe that
 $(\overline{X}, \overline{f^{n}})$ is the same as
 $(\overline{X}, \overline{f}^{n})$.

\begin{theorem}
\label{2-5} \cite[$\S 5$]{W2} Let $(\overline{X}, \overline{f})$  be a 1-solenoid. Then there is an integer $n$ and
pre-solenoid $(X', f')$ such that $(\overline{X}, \overline{f^n})$ is conjugate to
$(\overline{X'}, \overline{f'})$ and $X'$
has a single vertex That is, $X'$ is a wedge of cricles.
\end{theorem}

The pre-solenoid $(X', f')$ is usually called an
elementary presentation for the solenoid
 $(\overline{X}, \overline{f^n})$. We will
 usually denote the single vertex by $p$.

We begin our analysis of a pre-solenoid, $(X, f)$,
 having a single
vertex by observing that $f^{-1}\{ p \}$ is a finite
subset of $X$ and removing these points then
divides the edges of $X$ into a finite collection
of edges.

\begin{lemma} \cite{Y}
\label{2-6}
 Suppose that
$(X,f)$ is a pre-solenoid with a single
vertex $p$. Let $E=\{e_1,...e_m\}$ be
the edge set of $X$ with a given orientation.
 For each edge $e_i \in E$, we can give
 $e_i - f^{-1}\{ p \}$ the partition
  $\{e_{i,j}\}, 1\leq j\leq j(i)$,
satisfying the following
 \begin{enumerate}
\item
 the initial point of  $e_{i,1}$ is the
 initial point  of $e_i$.
\item
the terminal point of $e_{i,j}$ is the
initial point of $e_{i,j+1}$ for  $1\leq j < j(i)$,
\item
the terminal point of $e_{i,j(i)}$
is the terminal point of  $e_i$, and
\item
there is (with a small abuse of notation)
$1 \leq f(i,j) \leq n$ such that
$f\mid_{e_{i,j}}$ maps $e_{i,j}$ homeomorphically to
$e_{f(i,j)}$. We also set $s(i,j)$ to be $\pm 1$ according to whether $f\mid_{e_{i,j}}$ preserves or
reverses orientation.
\end{enumerate}
\end{lemma}

There is a convenient notation to describe
our pre-solenoids. We let $E^{*}$ denote
the set of words on the set $E$ and their inverses.
If we are given a function
$\widetilde{f}:E \rightarrow E^*$
we regard this as a wrapping rule as follows.
Fix $1 \leq i \leq m$. If we
 $$\widetilde{f}(e_i) = e^{s(i,1)}_{f(i,1)}e^{s(i,2)}_{f(i,2)}...e^{s(i,j(i))}_{f(i,j(i))}$$
then the interval $e_{i}$ is divided into $j(i)$
consecutive subintervals, $e_{i,j}$, and the map
$f$ carries $e_{i,j}$ homemorphically to $e_{f(i,j)}$,
either preserving or reversing the orientation
according to $s(i,j)$.

 \begin{example}
 \label{2-7} Let $X$ be a wedge of two clockwise circles $a , b$ with a unique vertex $p$ and  $f, g, k , h $  given by the wrapping rules: $\widetilde{f}:a\rightarrow aab , b \rightarrow abb$, $\widetilde{g}:  a\rightarrow  a^{-1}a^{-1} b^{-1}, b\rightarrow  a^{-1}b^{-1} b^{-1}$,  $\widetilde{k}:a\rightarrow  b^{-1}aa, b \rightarrow a^{-1}bb $  and $\widetilde{h}: a\rightarrow  a^{-1}ba, b\rightarrow  b^{-1}ab$. Then  $(X,f) $ and $ (X,k)$   are positively oriented  pre-solenoids,  $(X,g)$ is a negatively oriented  pre-solenoid and   $ (X,h)$ is not an oriented pre-solenoid. Figures 1 and 2 show these  pre-solenoids. Note that $(X,k)$ is also a  pre-solenoid  given in Figure 3 in which the orientation of edge $b$ has been  reversed.
\end{example}

\begin{figure}[h]
         \centering
         \begin{subfigure}[b]{0.4\textwidth}
                 \centering
                 \includegraphics[width=\textwidth]{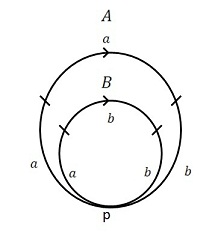}
                 \caption{Pre-solenoid (X,f)}
         \end{subfigure}%
         ~ 
         \begin{subfigure}[b]{0.4\textwidth}
                 \centering
                 \includegraphics[width=\textwidth]{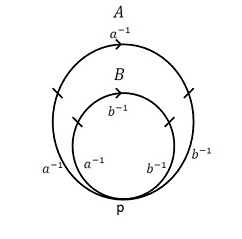}
                 \caption{Pre-solenoid (X,g)}
                 \label{fig:tiger}
         \end{subfigure}
         ~ 

         \caption{}\label{fig:animals}
 \end{figure}

 \begin{figure}[h]
         \centering
         \begin{subfigure}[b]{0.4\textwidth}
                 \centering
                 \includegraphics[width=\textwidth]{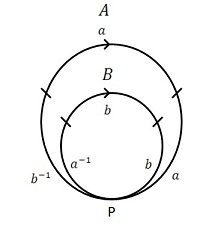}
                 \caption{Pre-solenoid (X,k)}
                 \label{fig:gull}
         \end{subfigure}%
         ~ 
         \begin{subfigure}[b]{0.4\textwidth}
                 \centering
                 \includegraphics[width=\textwidth]{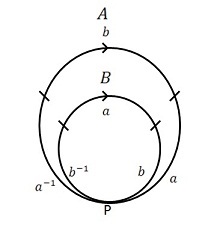}
                 \caption{Pre-solenoid (X,h)}
                 \label{fig:tiger}
         \end{subfigure}
         ~ 

        \caption{}\label{}
 \end{figure}

As a consequence of the expanding condition,
and by replacing $(X,f)$ by $(X, f^{n})$, for
some $n \geq 1$ if necessary, we may assume that
$j(i) \geq 3$, for all $i$.

We next let $U_{p}$ be a neighbourhood of $p$,
sufficiently small so that $f(U_{p})$ is
contained
in $ \cup_{i}  (e_{i,1} \cup e_{i, j(i)})$,
which is a neighbourhood of $p$, also.
From the fact that the flattening number is one,
$f(U_{p})$ is contained in the union of exactly
two of these intervals (and no fewer).
First suppose, these two are contained in
$e_{a}$ and $e_{b}$,
with $ a \neq b$, then by simply reversing
the orientations on these intervals as
needed, we may assume that the two intervals
 are $e_{a,1}$ and $e_{b, j(b)}$. The other case to
 consider is when the two are both in the same
 $e_{a}$, but then they must be $e_{a,1}$ and
 $e_{a, j(a)}$. In any case, we have $a$ and $b$
 such that
 $f(U_{p}) \subset e_{a,1} \cup e_{b,j(b)}$,
 allowing the possibility that $a = b$.

It follows for every $i$ that the two
sets  $f(e_{i,1} \cap U_{p})$ and
 $f( e_{i,j(i)} \cap U_{p})$
are  contained in one of $e_{a,1}$ or $e_{b,j(b)}$.
It also follows from the non-flattening condition
that $f(e_{a,1} \cap U_{p})$ and
$f( e_{b,j(b)} \cap U_{p})$ cannot be
contained in the same one. Replacing $f$
 by $f^{2}$ if necessary, we can assume that
 $f(e_{a,1} \cap U_{p}) ) \subset e_{a,1}$
 and  $f( e_{b,j(b)} \cap U_{p}) \subset e_{b,j(b)}$.

Notice that if $f(e_{i,j} \cap U_{p}) \subset e_{a,1}$, then we
have $f(e_{i,j}) = e_{a}$ and an analogous result holds for
$e_{b, j(b)}$.  In particular, we have $f(e_{a,1}) = e_{a}$
and $f(e_{b, j(b)}) = e_{b}$.

For each $i, 1 \leq j < j(i)$, we let $x_{i,j}$ denote
the unique terminal point of
$e_{i,j}$ which is also the initial point
of  $e_{i,j+1}$. Notice that the set of all such
$x_{i,j}$ is exactly $f^{-1}\{ p \} - \{ p \}$. Let $U_{i,j}$
be a neighbourhood of this point such that
$f(U_{i,j}) \subset U_{p}$.
Notice then that $f^{2}( e_{i,j} \cap U_{i,j})$ and
$f^{2}(e_{i,j} \cap U_{i,j+1})$ are each contained
in $e_{a,1}$ or in $e_{b, j(b)}$. Moreover, due to
the flattening condition, they cannot both be contained
in the same one.

Consider the following sets:
\begin{eqnarray*}
E_{a} & = & \{ e_{i} \in E \mid f(e_{i,1} \cap U_{p}),
f(e_{i,j(i)} \cap U_{p}) \subset e_{a,1} \} \\
E_{b} & = & \{ e_{i} \in E \mid f(e_{i,1} \cap U_{p}),
f(e_{i,j(i)} \cap U_{p}) \subset e_{b,j(b)} \} \\
E_{0}&  =  &  E - E_{a} - E_{b} \\
\end{eqnarray*}

\begin{theorem}
\label{2-8}
Let $(X, f)$ be a pre-solenoid
with $e_{a,1}, e_{b,j(b)}$ as above.
The sets $E_{a}$ and $E_{b}$ are disjoint.
Moreover,  $(X,f)$ is orientable if and only if
the set
$E_{a} \cup E_{b}$
is empty.
\end{theorem}

\begin{proof} The first statement is clear
since $j(b) \geq 3$ means that
$e_{a,1} \neq e_{b,j(b)}$.
First suppose that $E_{a} \cup E_{b}$ is
empty. The edges $e_{a}$ and $e_{b}$ have already been
given orientations when we define $e_{a,1}$ and
$e_{b, j(b)}$ (and these are consistent when $a=b$).
Consider any $e_{i}$ in $E$ and we assume it has been
given some orientation (so that $e_{i,1}$ and
$e_{i, j(i)}$ are defined).
By hypothesis,
$E_{a}  \cup E_{b}$ is empty and this means that
$e_{i}$ is in $E_{0}$. If we consider the two sets
$f(e_{i,1} \cap U_{p}), f(e_{i,j(i)} \cap U_{p})$,
one is contained in $e_{a,1}$ and the other in $e_{b,j(b)}$.
By reversing the orientation of $e_{i}$ if necessary, we
may assume that  $f(e_{i,1} \cap U_{p}) \subset e_{a,1}$,
while $f(e_{i,j(i)} \cap U_{p}) \subset e_{b,j(b)}$.

We claim that the pre-solenoid is now positively oriented.
We will show that each $e_{i,j}, 1 \leq j < j(i)$
is mapped in an orientation preserving way to
$f(e_{i,j})$. We proceed by induction on $j$.
The case of $j=1$ is true simply by our choice of
orientation on $e_{i}$;  $e_{i,1} \cap U_{p}$
is mapped into $e_{a,1}$, which is the initial segment
of $e_{a} = f(e_{i,1})$. Assume the statement is
 true for some given $i,j$ and let $e_{i'} = f(e_{i,j})$.
  Let $V$ be a neighbourhood of
 the boundary point between $e_{i,j}$ and $e_{i,j+1}$
 sufficiently small so that $f(V) \subset U_{p}$.
Then  $e_{i,j} \cap V$ is being mapped by $f$
into $  e_{i',j(i')} \cap U_{p}$, by induction
hypothesis. By the choice of orientation on $e_{i'}$,
we know that $f(e_{i', 1)} \cap U_{p})$ is contained in $e_{a,1}$
and as $e_{i'}$ is in $E_{0}$, $f(e_{i', j(i'))} \cap U_{p})$
is contained in $e_{b,j(b)}$. By the non-folding hypothesis,
$f^{2}(e_{i,j+1} \cap V)$ cannot be contained in $e_{b,j(b)}$.
It follows then that it is contained in $e_{a,1}$.
This then implies, letting $f(e_{i,j+1})= e_{i''}$, that
$f(e_{i,j+1} \cap V)$ cannot be contained in
$e_{i'',j(i'')}$. The conclusion of the induction
statement follows.

Conversely suppose that the 1-solenoid  $(X,f)$
  is oriented. It follows immediately that for all $i$,
  $f(e_{i,1} \cap U_{p}) \subset e_{a,1}$ while
  $f(e_{i,j(i)} \cap U_{p}) \subset e_{b,j(b)}$. But this means
  that $e_{i}$ is in $E_{0}$.
 \end{proof}

Following Yi \cite{Y}, for a pre-solenoid $(X, f)$, we
 define a graph
$G_{X} $ (or simply $G$ if no confusion will arise)
by setting $G_{X}^{0} = E = \{ e_{i} \mid 1 \leq i \leq n \}$
 and
 $G_{X}^{1} = \mathcal{A} = \{ e_{i,j} \mid 1 \leq i \leq m, 1 \leq j \leq j(i) \}$.
The initial and terminal maps are given  by
$i(e_{i,j}) = e_{i}, t(e_{i,j}) = f(e_{i,j})= e_{f(i,j)}$,
for each suitable $i$ and $j$.
For simplicity, we let $(\Sigma_{X}, \sigma)$ denote
the associated shift of finite type.

The shifts of finite type $(\Sigma_X,\sigma_X)$  associated to  all of the solenoids defined in the above example  are  represented  by the same   graph $G$  given  by   part (B) of Figure 3 and $$D^s(\Bbb ZG^0)=D^u(\Bbb ZG^0)=\{ (i,i+j) : \ i\in \Bbb Z[1/3], j \in \Bbb Z \}$$
 by
\[\left\{ {\begin{array}{*{20}{c}}
{{D^s}({\Bbb ZG^0}) \to \{ (i,i + j):i \in \Bbb Z[1/3],j \in \Bbb Z  \} ,{\rm{  }}}\\
{[{\upsilon _1} +  {\upsilon _2},k] \to ({3^{ - (k-1)}},{3^{ - (k-1)}}),{\rm{ }}}\\
{[{\upsilon _1} - {\upsilon _2},l] \to (-1,1).}
\end{array}} \right.\]\\
\begin{figure}
         \centering
         \begin{subfigure}[b]{0.4\textwidth}
                 \centering
                 \includegraphics[width=\textwidth]{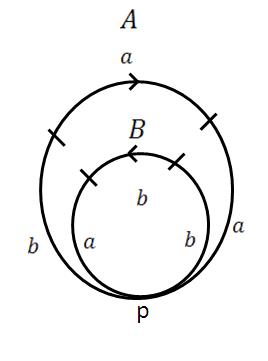}
                 \caption{Pre-solenoid (X,k) with\\ the new orientation}
                 \label{fig:gull}
         \end{subfigure}%
         ~ 
         \begin{subfigure}[b]{0.6\textwidth}
                 \centering
                 \includegraphics[width=\textwidth]{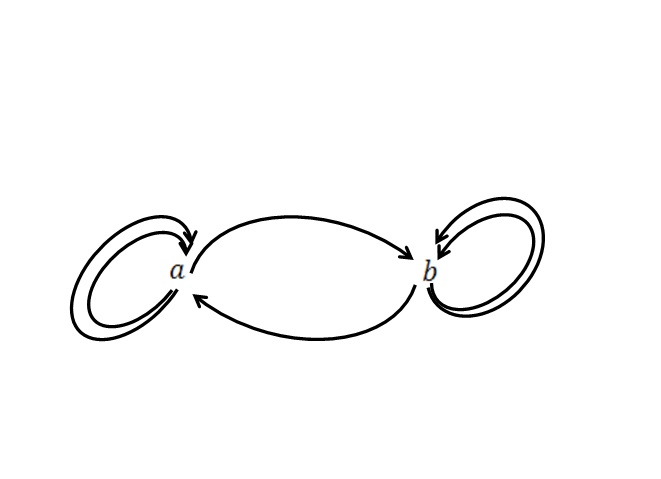}
                 \begin{flushright}
                      \quad \caption{Graph G presenting the  shift of finite
                                            type \\   \quad \ \ \  correspond to the solenoids defined in example \ref{2-7}}
                 \end{flushright}
                 \label{fig:tiger}
         \end{subfigure}
         ~ 

       \caption{}\label{fig:animals}
 \end{figure}

The sets $E_{a}$ and $E_{b}$
 we defined above provide us with specific
elements of the dimension groups $D^{s}(\Sigma_{X}, \sigma)$ and  $D^{u}(\Sigma_{X}, \sigma)$ associated with
$(\Sigma_{X}, \sigma)$.

\begin{lemma}
\label{2-9}
Let $(X, f)$ be a pre-solenoid as above.
Define $w$ in $\mathbb{Z} G^{0}$ by
$w = Sum(E_{a} ) -   Sum(E_{b} )$.
Also, let $w^{*}: \mathbb{Z} G^{0} \rightarrow \mathbb{Z}$
be the group homomorphism
 which sends each element of $E_{a}$ to $1$, each
 element of $E_{b}$ to $-1$ and the elements of
$E_{0}$ to zero. We have the following:
 \begin{enumerate}
 \item $w=0$ if and only if
 $(X, f)$ is orientable,
 \item $ \gamma_{G}^{s}(w) = w$,
\item $\mathbb{Z} G^{0}/ < w>$ is torsion free, where
$<w>$ denotes the cyclic subgroup generated by $w$,
\item $w^{*}=0$ if and only if $(X,f)$ is orientable,
\item $w^{*} \circ \gamma_{G}^{u} = w^{*}$.
\end{enumerate}
\end{lemma}

\begin{proof}
The first and fourth statements follow immediately
from Theorem \ref{2-8}.

We compute $\gamma_{G}^{s}(w)$. Write
$\gamma_{G}^{s} = \sum_{i} k_{i} e_{i}$, for some
choice of integers $k_{i}$. For a fixed $k_{i}$, it
follows from the definitions of $\gamma_{G}^{s}$ and
$G$ that $k_{i}$ is the number of $1 \leq j \leq j(i)$
with $f(e_{i,j}) \in E_{a}$ minus the number
of such $j$ with $f(e_{i,j}) = E_{b}$.

Construct $\alpha$ in $\{ a, b \}^{2j(i)}$ by first
considering the sequence of open sets
\[
e_{i,1} \cap U_{p}, e_{i,1} \cap U_{i,1}, e_{i,2} \cap U_{i,1}
\ldots, e_{i,j(i)} \cap U_{i,j(i)-1}, e_{i,j(i)} \cap U_{p}.
\]
To each we apply $f^{2}$ and obtain a sequence of sets,
each being contained in either $e_{a,1}$ or
$e_{b, j(b)}$. The sequence of values of $a$ or $b$
are obtained accordingly. It follows from the
non-folding condition that for every $j$,
$\alpha_{2j} \neq \alpha_{2j+1}$. Also,
$f(e_{i,j}) $ is in $E_{a}$ if and only if
$\alpha_{2j-1} = \alpha_{2j} = a$ and is in
$E_{b}$ if and only if $\alpha_{2j-1} = \alpha_{2j} = b$.
So $k_{i}$ is the number of consecutive $a$'s minus the number
of consecutive $b$'s.

We claim that if $(\alpha_{2j}, \alpha_{2j+1}) = (b,a)$ for some
$j$, we may change it to $(a,b)$ without altering the value
$k_{i}$. There are four cases to consider, depending on the values
of $\alpha_{2j-1}$ and $\alpha_{2j+2}$. First, suppose they are both $a$. In this  case,
$(\alpha_{2j-1}, \alpha_{2j}, \alpha_{2j+1}, \alpha_{2j+2}) =(a,b,a,a)$ and contains exactly one pair of adjacent $a$'s and no adjacent $b$'s. Switching the places of the two central entries
does not change this fact and so does not alter $k_{i}$.
Next, we suppose that $(\alpha_{2j-1}, \alpha_{2j}, \alpha_{2j+1}, \alpha_{2j+2}) =(a,b,a,b)$. Here, we have no adjacent $a$'s or
$b$'s. Switching the two central entries results in a pair
of adjacent $a$'s and a pair of adjacent $b$'s, but the difference
is still zero. There are two other cases which are done similarly
and we leave to the reader.

Now we may assume that
 $(\alpha_{2j}, \alpha_{2j+1}) = (a,b)$ for all
$j$.
Now we consider the possible values of
$\alpha_{1}$ and $\alpha_{2j(i)}$.
If both are $a$, then there is exactly one pair
of adjacent $a$'s ($\alpha_{1}, \alpha_{2}$) and no adjacent $b$'s.
This means that $k_{i}=1$. On the other hand, it also
follows that $e_{i}$ is in $E_{a}$ in this case.
If $\alpha_{1}=a$ and $\alpha_{2j(i)}=b$, then there is one
pair of adjacent $a$'s and one pair of adjacent $b$'s.
If $\alpha_{1}=b$ and $\alpha_{2j(i)}=a$, then there are
 adjacent $a$'s or adjacent $b$'s. In both cases, we have
 $k_{i}=0$. But also in these cases, we see that
 $e_{i}$ is in $E_{0}$. Finally, in the case
 that $\alpha_{1}=\alpha_{2j(i)}=b$, we see that
 $k_{i}=-1$ and that $e_{i}$ is in $E_{b}$. We have now proved that
 $\gamma^{s}_{G}(w)=w$.

The third statement is clear. For the fifth, we can
regard $w$ as a group  homomorphism from $\mathbb{Z}$
 to $\mathbb{Z} G^{0}$.
Then the map $w^{*}$ is simply the dual of this map and
$\gamma_{G}^{u}$ is simply the dual of $\gamma^{s}_{G}$.
By dual, we mean to replace a group $H$ by $Hom(H, \mathbb{Z})$.
Moreover, we identify
$\mathbb{Z}A$ and $Hom(\mathbb{Z}A, \mathbb{Z})$, for
any set $A$, by the canonical isomorphism.
In  this way, the fifth
 statement simply follows from the first.
\end{proof}

In the case where the elements $w$ and $w^{*}$
 are non-zero, they also provide elements
 having analogous properties at the level
 of the inductive limit groups.
The following is an immediate consequence
of the last lemma and we omit the proof.

\begin{lemma}
\label{2-10}
Suppose that $(X, f)$ is a non-orientable
pre-solenoid satisfying
the conditions above. If we  identify $D^{s}(G_{X})$
with
$D^{s}(\Sigma_{X}, \sigma)$
and $D^{u}(G_{X})$ with
$D^{u}(\Sigma_{X}, \sigma)$, then we have the following.
\begin{enumerate}
\item
For all $n$, $[w, n] = [w, 1] \neq 0$ in
$D^{s}(\Sigma_{X}, \sigma)$ and the quotient group
$D^{s}(\Sigma_{X}, \sigma)/ <[w,1]>$ is torsion free.
\item
The map $w^{*}: D^{u}(\Sigma, \sigma) \rightarrow \mathbb{Z}$ defined by $w^{*}[a,n] = w^{*}(a)$,
for $a \in  \mathcal{A}$ is well-defined
and surjective.
\end{enumerate}
\end{lemma}

We define two  maps
\begin{eqnarray*}
\rho: (\Sigma_{X}, \sigma) & \rightarrow &  (X, f), \\
\overline{\rho}: (\Sigma_{X}, \sigma) & \rightarrow &  (\overline{X}, \overline{f})
\end{eqnarray*}
as follows.
For each point  $I= (I_{n})_{n \in \mathbb{Z}} \in \Sigma_X$, let
 \begin{eqnarray}\label{6}
    \{ \rho(I) \} & =  &  \bigcap^{\infty}_{N=0}
    \overline{\cap_{n=0}^{N} f^{-n}(I_{n})} \\
    \overline{\rho}(I) & = & (\rho(I), \rho \circ \sigma^{-1}(I),  \rho \circ \sigma^{-2}(I), \ldots )
 \end{eqnarray}

 To see that the set in the definition of
 $\rho(I)$ is a singleton, observe that the intersection
 $\cap_{n=0}^{N} f^{-n}(I_{n})$ is an open interval
 in $I_{0}$ and its closure is a closed interval.
 Moreover, the lengths of these intervals
 decrease geometrically with $N$. It follows that the
 intersection is a single point. (The definition
 looks a little unusual. This is essentially due to the fact that
 we are trying to use the sets $\pi^{-1}(e_{i})$ as a Markov
 partition, but they are slightly too large; the ends wrap
 around and meet at the vertex.)
   It is easy to verify that $\overline{\rho}$ is a factor map, that
   $\rho \circ \sigma = f \circ \rho$,
   and obviously
    $\rho = \pi \circ \overline{\rho}$.

We need to describe specific properties of the map
$\overline{\rho}: \Sigma_{X} \rightarrow \overline{X}$.
We begin with the following rather technical lemma. The proof
is essentially found in Yi \cite{Y}, but we provide a proof
here for completeness and because we will require some more
precise information about $\rho$. The point is that
our homology computations to be done later
 will require precise knowledge
about the points $I \neq J$ in $\Sigma_{X}$ with
$\overline{\rho}(I) = \overline{\rho}(J)$.  We will let $\mathcal{A}_{0} = \{ e_{i,1}, e_{i,j(i)} \mid 1 \leq i \leq m\}$.

\begin{lemma}
\label{2-11}
Let $(X,f)$ be a pre-solenoid
with unique vertex $p$ as above.
\begin{enumerate}
\item
Suppose that
$I=(I_{n})_{n \in \mathbb{Z}}$ and
$J=(J_{n})_{n \in \mathbb{Z}}$ in $\Sigma_{X}$
are such that
\begin{enumerate}
\item $\overline{\rho}(I) = \overline{\rho}(J)$,
\item $I_{0} \neq J_{0}$, and
\item $\rho(I) = \rho(J) \neq p $.
\end{enumerate}
Then, up to interchanging $I$ and $J$, we have
\begin{enumerate}
\item  $\rho(I) = \rho(J) \in f^{-1}\{ p \} - \{ p \}$,
\item $I_{n} = J_{n}$, for all $n < 0$,
\item $\{ I_{0}, J_{0} \} = \{ e_{i,j}, e_{i,j+1} \}$, for
some $i, 1 \leq j < j(i)$,
\item  $I_{1}$ and $J_{1}$ are both in $\mathcal{A}_{0}$
with $f(I_{1} \cap U_{p}) \subset e_{a,1}$
and $f(J_{1} \cap U_{p}) \subset e_{b, j(b)}$,
\item
$I_{n} = e_{a,1}, J_{n} = e_{b,j(b)}$, for all $n \geq 2$.
\end{enumerate}
\item
If $e \neq e_{a,1}, e_{b,j(b)}$
is any element of $\mathcal{A}_{0}$ with
$f(e \cap U_{p}) \subset e_{a,1}$, then there
exists $I \neq J$ satisfying the conclusion of
the first part.
If $e \neq e_{a,1}, e_{b,j(b)}$
 is any element of $\mathcal{A}_{0}$ with
$f(e \cap U_{p}) \subset e_{b,j(b)}$, then there
exists $I \neq J$ satisfying the conclusion of
the first part.
\end{enumerate}
\end{lemma}

\begin{proof}
 We know that
$\rho(I) = \rho(J)$ is in
the closure of the sets $I_{0}$ and
 $J_{0}$. On the other hand, $I_{0}$ and $J_{0}$
are disjoint since they are unequal. It follows that
$\rho(I)$ is a boundary point
of each and is therefore in $f^{-1}\{ p \}$.
Since we assume that
$\rho(I) \neq p$, it is in the closure
of exactly two elements of $\mathcal{A}$.
More specifically, there
exist a unique $i, 1 \leq j < j(i)$ with
$ \{ I_{0}, J_{0} \} = \{ e_{i,j}, e_{i, j+1} \}$.

We now want to show that $I_{n} = J_{n} $,
for all $n < 0$.
We know that \newline
$f^{-n}(\rho(\sigma^{n}(I)))
= \rho(I) \neq p$
and so $\rho(\sigma^{n}(I)) $ is in the interior of
a unique element of $\mathcal{A}$. On the other hand,
it is also in the closure of
\[
\cap_{i=0}^{-n} f^{-i}(\sigma^{n}(I)_{i})  =
\cap_{i=0}^{-n} f^{-i}(I_{i+n})
  \subset I_{n}.
\]
This means that $I_{n}$ is the unique element of
$\mathcal{A}$ containing $\rho(\sigma_{n}(I)) $.
The same argument applies to $J$ and using
$\rho(\sigma^{n}(I))= \rho(\sigma^{n}(J))  $, we
conclude that $I_{n} = J_{n}$.

Next, we claim that for any integer $M \geq 2$, we have
$\{ I_{M}, J_{M} \} = \{ e_{a,1}, e_{b, j(b)} \}$ and
$I_{M+1} = I_{M}, J_{M+1} = J_{M}$.
Let $V$ be a neighbourhood
of $\rho(I)$ such that $f^{m}(V) \subset U_{p}$ for all
$1 \leq m \leq M + 2$ and so that $f^{M+2}$ is injective
on $V$. We know that for $2 \leq m \leq M+2$,
 $f^{m}(I_{0} \cap V)$ and
$f^{m}(J_{0} \cap V)$ are each contained in a set of the form
$e_{a,1} \cap U_{p}$ or $e_{b,j(b)} \cap U_{p}$. They cannot
be contained in the same one, because of the non-folding axiom.
Therefore, up to switching $I$ and $J$, we have
 $f^{M}(I_{0} \cap V)  \subset e_{a,1} \cap U_{p}$
$f^{M}(J_{0} \cap V) \subset e_{b,j(b)} \cap U_{p}$. These imply
$f^{M+1}(I_{0} \cap V) \subset e_{a,1}$ and
$f^{M+1}(J_{0} \cap V) \subset e_{b,j(b)}$. Since $\rho(I)$
is defined as an intersection, we may choose
an integer $L > M+1$ such that
$\cap_{n=0}^{L} f^{-n}(I_{n}) \subset V$. This implies
that $f^{M}(I_{0} \cap V) \subset I_{M}$ and
$f^{M+1}(I_{0} \cap V) \subset I_{M+1}$. But since the intervals
of $\mathcal{A}$ are pairwise disjoint and the set
$f^{M}(I_{0} \cap V)$ is contained in both $I_{M}$ and $e_{a,1}$,
we conclude these must be equal. Similarly, we find
$I_{M+1} = e_{a,1}$, $J_{M} = J_{M+1} = e_{b, j(b)}$.

For the proof of the second part, we consider
the first statement only. We know that $f$
is surjective. Moreover, since the elements
of $\mathcal{A}_{0}$ all map
to $e_{a,1}$ or $e_{b,j(b)}$, we may find
some $i,j$ where $f(U_{i,j}) $ meets $e \cap U_{p}$.
One of $f(e_{i,j} \cap U_{i,j})$ and $f(e_{i,j+1} \cap U_{i,j})$
meets $e$. Let us assume it is the former.
Then
 $f^{2}(e_{i,j} \cap U_{i,j})$
 is contained in $e_{a,1}$ and, by the non-folding
 axiom, $f^{2}(e_{i,j+1} \cap U_{i,j})$ must be contained in
 $e_{b,j(b)}$. Let $I_{0} = e_{i,j}, J_{0} = e_{i,j+1}$.
 Let $I_{1} = e$ and $J_{1}$ be the unique element of
 $\mathcal{A}_{0}$ containing
 $f(e_{i,j+1} \cap U_{i,j})$. Let
 $I_{n}= e_{a,1}, J_{n}=e_{b,j(b)}$, for all $n \geq 2$.
 The elements $I_{n}=J_{n}, n < 0,$ may be chosen arbitrarily
 so that $f(I_{n-1}) \supset I_{n}$, for all $n$.
\end{proof}

\begin{theorem}
\label{2-12} Suppose $(X,f)$ be a pre-solenoid
with unique vertex $p$ as above.
\begin{enumerate}
\item If $I=(I_{n})_{n \in \mathbb{Z}}$ and
$J=(J_{n})_{n \in \mathbb{Z}}$ in $\Sigma_{X}$ are such that
$\overline{\rho}(I) = \overline{\rho}(J)$, then, up to interchanging $I$ and $J$,
one of the following holds.
\begin{enumerate}
\item $I=J$.
\item $I_{n} = e_{a,1}, J_{n} = e_{b,j(b)}$, for all integers $n$.
\item There is a unique $N$ such that $I_{N} \neq J_{N}$,
$I_{n} = J_{n}$, for all $n < N$ and
$I_{n}= e_{a,1}, J_{n} = e_{b, j(b)}$, for all $n \geq N + 2$. Moreover, $N$ is characterized by the
condition $\overline{\rho} \circ \sigma^{N}(I) \in f^{-1}\{ p \} - \{ p \}$.
\end{enumerate}
\item
The map $\overline{\rho}$ is one-to-one on   $\Sigma_X-\bigcup^\infty_{m=0} \overline{\rho}^{-1} \circ \pi^{-1} \circ f^{-m-1}\{ p \}$.
\item The map $ \overline{\rho}:(\Sigma_X, \sigma_X) \rightarrow (\overline{X},\overline{f})$ is at most two-to-one.
\item
The map $ \overline{\rho}:(\Sigma_X, \sigma_X)  \rightarrow (\overline{X},\overline{f})$ is  an $s$-bijective map.
\end{enumerate}
\end{theorem}

\begin{proof}
The Lemma \ref{2-11}
proves the first statement under the conditions $I \neq J$ and
$\overline{\rho} \circ \sigma^{N}(I) \neq p$, for some $N$.
It remains to consider the case $I \neq J$ and
$\overline{\rho} \circ \sigma^{N}(I) =p$, for all $N$.

But then for any $N$,
$I_{N} \cap f^{-1}(I_{N+1}) \cap f^{-2}(I_{N+2})$ must
contain $p$ in its closure. It follows that
this set is contained in either
$ e_{i,1}$ or $ e_{i,j(i)}$ for some
$i$. In the first case,
 $f(U_{p} \cap I_{N} \cap f^{-1}(I_{N+1}) \cap f^{-2}(I_{N+2}))$ is a non-empty subset of $f(U_{p} \cap e_{i,1}) \subset e_{a,1}$  and also
of $I_{N+1}$.  This implies
that $I_{N+1} = e_{a,1}$. In addition,
$f^{2}(U_{p} \cap I_{N} \cap f^{-1}(I_{N+1}) \cap f^{-2}(I_{N+2}))$ is a non-empty subset of $f^{2}(U_{p} \cap e_{i,1}) \subset e_{a,1}$  and also
of $I_{N+2}$. This implies
that $I_{N+2} = e_{a,1}$. In the latter case, the same argument proves that
$I_{N+1} = I_{N+2} = e_{b, j(b)}$. What
we have shown is that, for any $N$, $I_{N+1}$
is either $e_{a,1}$ or
$e_{b,j(b)}$ and $I_{N+1}=I_{N+2}$. The
conclusion follows since this holds for all $N$.

The second and third statements are immediate.
The fourth follows from the first: no non-trivial pair
$I$ and $J$ with $\overline{\rho}(I)  = \overline{\rho}(J)$
are stably equivalent. Thus,
$\overline{\rho}$ is $s$-resolving. The shift
$\Sigma_{X}$ is irreducible from the hypotheses on $(X, f)$, so
$\overline{\rho}$ is $s$-bijective by Theorem 2.5.8 of \cite{P1}.
\end{proof}

\section{Homology}

\subsection{Statement of the results}

We state our two main theorems as follows.

 \begin{theorem}
 \label{3-1}
  Let $(X, f)$ be a pre-solenoid
   and $(\overline{X},\overline{f})$ be its
   associated  one-solenoid. If $(X, f)$ is orientable, then
\begin{eqnarray*}
H^s_N(\overline{X},\overline{f})= \left \{\begin{array}{lll}
D^s(\Sigma_{X},\sigma) \  &  \ N=0,\\ \Bbb Z \ & \  N=1,\\ 0  \ &  \  N\neq 0,1. \end{array} \right.
\end{eqnarray*}

If $(X,f)$ is not orientable, then
\begin{eqnarray*}
H^s_N(\overline{X},\overline{f})= \left \{\begin{array}{lll}
D^s(\Sigma_{X},\sigma)/{<2[w,1]>} &  \ N=0,\\ 0  \ &  \  N\neq 0. \end{array} \right.
\end{eqnarray*}
\end{theorem}

\begin{theorem}
\label{3-2}
 Let $(X, f)$ be a pre-solenoid
   and $(\overline{X},\overline{f})$ be its
   associated  one-solenoid. If $(X, f)$
  is orientable, then
\begin{eqnarray*}
H^u_N(\overline{X}, \overline{f})= \left \{\begin{array}{lll}
D^u(\Sigma_X,\sigma) \  &  \ N=0,\\ \Bbb Z \ & \  N=1,\\ 0  \ &  \  N\neq 0,1. \end{array} \right.
\end{eqnarray*} If $(X, f)$ is not orientable, then
\begin{eqnarray*}
H^u_N(\overline{X},\overline{f}) = \left \{\begin{array}{lll}
Ker (w^*)  &  \ N=0,\\ \Bbb Z_2 \ & \  N=1,\\ 0  \ &  \  N\neq 0,1. \end{array} \right.
\end{eqnarray*}
\end{theorem}

The proof of these two results is quite long and we defer it
to the next subsection.
We note some consequences.

\begin{corollary}
\label{3-3}
Let $(X, f)$ be a pre-solenoid and
$(\overline{X},\overline{f})$ be its associated
 one-solenoid.
If $(X,f)$ is orientable then all of the  homology groups
of $(\overline{X}, \overline{f})$ are torsion free.

 If $(X, f)$ is not orientable, then we have
 \[
Tor(H^{s}_{0}(\overline{X},\overline{f})) \cong
Tor(H^{u}_{1}(\overline{X},\overline{f})) \cong
\mathbb{Z}_{2},
 \]
where $Tor(H)$ denotes the torsion subgroup of $H$,
 and the remaining homology groups are torsion free.
\end{corollary}

Notice, in particular, that this means that the notion
of orientability is independent of the choice of
$(X,f)$. This provides a new proof of Theorem
\ref{2-4}.

 \begin{example}
 \label{3-4}
 Suppose $(X,f), (X,g) \ and \  (X,k)  $  be the pre-solenoids defined in example \ref{2-7}, then  $H^s_N(\overline{X},\overline{f})=H^s_N(\overline{X},\overline{g})=H^s_N(\overline{X},\overline{k})$ for each  $N \geq 0$
 \begin{eqnarray*}
H^s_N(\overline{X},\overline{f})= \left \{\begin{array}{lll}
\{ (i,i + j):i \in \Bbb Z[1/3],j \in \Bbb Z  \} \  &  \ N=0,\\ \Bbb Z \ & \  N=1,\\ 0  \ &  \  N\neq 0,1. \end{array} \right.
\end{eqnarray*}

 Also for the 1-solenoid $(\overline{X},\overline{h})$ defined in that example:

\begin{eqnarray*}
H^s_N(\overline{X},\overline{h})= \left \{\begin{array}{lll}
\{ (i,i + j):i \in \Bbb Z[1/3],j \in \Bbb  Z  \} / 2\Bbb Z(-1,1) &  \ N=0,\\ 0  \ &  \  N\neq 0. \end{array} \right.
\end{eqnarray*}
\end{example}

\begin{example}
\label{3-5}
 Suppose $(X,f), (X,g)$ and $  (X,k)  $  are the pre-solenoids defined in Example \ref{2-7}, then $H^u_N(\overline{X},\overline{f})=H^u_N(\overline{X},\overline{g})=H^u_N(\overline{X},\overline{k})$
 \begin{eqnarray*}
H^u_N(\overline{X},\overline{f})= \left \{\begin{array}{lll}
\{ (i,i + j):i \in \Bbb Z[1/3],j \in \Bbb Z  \} \  &  \ N=0,\\ \Bbb Z \ & \  N=1,\\ 0  \ &  \  N\neq 0,1. \end{array} \right.
\end{eqnarray*}

 And  also for the pre-solenoid $(X,h)$ defined  in that example:

\begin{eqnarray*}
H^u_N(\overline{X},\overline{h})= \left \{\begin{array}{lll}
\Bbb Z[1/3] \  &  \ N=0,\\ \Bbb Z_2 \ & \  N=1,\\ 0  \ &  \  N\neq 0,1. \end{array} \right.
\end{eqnarray*}
\end{example}

The Cech cohomology of the space $\overline{X}$
may be computed as follows.
First, $\overline{X}$   is written as the inverse
limit of the stationary system with space
$X$ and map $f$ and so its cohomology is
the direct limit of the stationary system of groups
$\check{H}^{*}(X)$ with maps $f^{*}$. Since $X$ is
the wedge of $m$ circles, its
cohomology is $\mathbb{Z}$ in dimension zero,
$\mathbb{Z}^{m}$ in dimension one and zero
in all other dimensions. Moreover, a simple direct
computation shows that in the special case that $(X,f)$
is orientable, then the map $f^{*}$ is the identity
 in dimension zero and agrees with
$\gamma^{u}_{G}$ in dimension one if we identify
$\mathbb{Z} G^{0}$ with $\mathbb{Z}^{m}$
in an obvious way. Therefore, we have the following.

\begin{theorem}
\label{3-6}
If $(\overline{X},\overline{f})$  is
an orientable one-solenoid then
\begin{eqnarray*}
H^{u}_{0}(\overline{X}, \overline{f})
& \cong  & \check{H}^{1}(\overline{X}) \\
 H^{u}_{1}(\overline{X}, \overline{f})
& \cong  & \check{H}^{0}(\overline{X}).
\end{eqnarray*}
\end{theorem}

On the other hand, if $(X, f)$ is not orientable, then
the map $f^{*}$ notices the difference
in orientations while $\gamma^{u}_{G}$ does not.
In particular, the map $h$ of Example \ref{2-7}
induces an isomorphism
on $\check{H}^{*}(X)$ and so we have
 $\check{H}^{1}(\overline{X}) \cong \mathbb{Z}^{2}$
 and hence this group is
not isomorphic to $H^{u}_{0}(\overline{X}, \overline{h})$.

\subsection{Proofs}

In general, the computation of the
homology groups \newline
$H^{s}_{N}(X, \varphi), H^{s}_{N}(X, \varphi), N \in \mathbb{Z}$,
 for a Smale space $(X, \varphi)$
is a rather complicated  business involving
 double complexes. However, we may
 appeal to two special features in our case.
 The first is that our solenoid is the image
 of a shift of finite type under an $s$-bijective
 factor map $\rho$. A general investigation
 of such Smale spaces can be found in Wieler \cite{W4}.
 This reduces the double
 complexes to usual complexes indexed by the
 integers. The second feature is that the map
 $\rho$ is at most two-to-one.
 This means that there are only two non-zero
 entries in the complex. Specifically,
 Theorem 4.2.12 and Theorem 7.2.1 of \cite{P1}
 provide us with
 the following description.

\begin{theorem}\cite{P1}
\label{3-7}
Let $(X,f)$ be a pre-solenoid, $(\overline{X},\overline{f} )$ its associated one-solenoid
and  $\overline{\rho}:(\Sigma_{X},\sigma)\rightarrow (\overline{X},\overline{f} )$ be the $s$-bijective factor map of the last section.
\begin{enumerate}
\item
The homology  $ H^s_N(\overline{X},\overline{f} )$ is
 naturally isomorphic to the homology of the
 complex $(D^s_{\mathcal{Q}}(\Sigma_*(\overline{\rho})), d^s(\overline{\rho}))$. Similarly,
 the homology $ H^u_N(\overline{X},\overline{f} )$
  is naturally isomorphic to the homology of the
 complex $(D^u_{\A}(\Sigma_*(\overline{\rho})), d^{u*}(\rho))$.
\item The only non-zero terms in the
complexes $(D^s_{\mathcal{Q}}(\Sigma_*(\overline{\rho})), d^s(\overline{\rho}))$
\newline
and
$(D^u_{\A}(\Sigma_*(\overline{\rho})), d^{u*}(\overline{\rho}))$ occur
in entries $N = 0$ and $N = 1$.
\item
For $N \neq 0, 1$, we have
\[
H^s_N(\overline{X},\overline{f} ) =
H^u_N(\overline{X},\overline{f} ) = 0.
\]
\item
We have
\[
\begin{array}{rcl}
H^s_0(\overline{X},\overline{f} ) & = &
coker \left( d^s(\overline{\rho})_{1}:
D^s_{\mathcal{Q}}(\Sigma_{1}(\overline{\rho}))
\rightarrow D^s_{\mathcal{Q}}(\Sigma_{0}(\rho)) \right) \\
H^s_1(\overline{X},\overline{f} ) & = &
ker \left( d^s(\overline{\rho})_{1}:
 D^s_{\mathcal{Q}}(\Sigma_{1}(\overline{\rho}))
\rightarrow D^s_{\mathcal{Q}}(\Sigma_{0}(\overline{\rho})) \right) \\
H^u_0(\overline{X},\overline{f} ) & = &
ker \left( d^{u*}(\overline{\rho})_{1}:
D^s_\A(\Sigma_{0}(\overline{\rho}))
\rightarrow D^s_\A(\Sigma_{1}(\overline{\rho})) \right) \\
H^u_1(\overline{X},\overline{f} ) & = &
coker \left( d^{u*}(\overline{\rho})_{1}:
 D^s_\A(\Sigma_{0}(\overline{\rho}))
\rightarrow D^s_\A(\Sigma_{1}(\overline{\rho})) \right)
\end{array}
\]
\end{enumerate}
\end{theorem}

Our first task is to identify the two groups
involved in the two complexes.
It is a general fact that the groups in position
$0$ are the simplest. We have
\[
\begin{array}{ccccc}
D^s_{\mathcal{Q}}(\Sigma_{0}(\overline{\rho})) & = &
 D^{s}(\Sigma_{X}, \sigma) & = & D^{s}(G^{L}), \\
D^u_\A(\Sigma_{0}(\overline{\rho})) &  = &
D^{u}(\Sigma_{X}, \sigma) & = & D^{u}(G^{L}),
\end{array}
\]
for any $L \geq 1$.

Next, we turn to the groups in position 1 of our complexes.
 Notice that
we write $e_{a,1}^{L}$ for
 $(e_{a,1}, e_{a,1}, \ldots, e_{a,1})$ and $e^{L}_{b,j(b)}$ for
$(e_{b,j(b)}, e_{b,j(b)}, \ldots ,e_{b,j(b)})$ in
$G^{L}$. for $L \geq 1$.

\begin{lemma}
\label{3-8}
Let $L$ be any integer such that
\[
\overline{\rho}: (\Sigma_{G_{L}}, \sigma)
\rightarrow (\overline{X}, \overline{f})
\]
 is  regular (see \cite{P1}).
\begin{enumerate}
\item
Then $D^s_{\mathcal{Q}}(\Sigma_{1}(\overline{\rho})) =  D^{s}_Q(G^{L}_{1})$
 is
an infinite cyclic group with generator
$Q[(e_{a,1}^{L-1}, e_{b, j(b)}^{L-1}), 1]
=  Q[(e_{a,1}^{L-1}, e_{b, j(b)}^{L-1}), n]$,
for all $n \geq 1$.
\item
Then $D^u_\mathcal{A}(\Sigma_{1}(\overline{\rho})) =  D^{u}_{\mathcal{A}}(G^{L}_{1})$ is
an infinite cyclic group with generator
$[(e_{a,1}^{L-1}, e_{b,j(b)}^{L-1}) - (e_{b,j(b)}^{L-1}, e_{a,1}^{L-1}), 1]  =  [(e_{a,1}^{L-1}, e_{b,j(b)}^{L-1}) - (e_{b,j(b)}^{L-1}, e_{a,1}^{L-1}), n]$, for all $n \geq 1$.
\end{enumerate}
\end{lemma}

\begin{proof}
Theorem \ref{2-12} gives us a complete description
of the elements of $G^{L}_{1}$; they
are simply restrictions of the infinite paths
$I,J$ described there to intervals of length $L$.

Let $(I,J)$ be in $G^{L}_{1}$. If $I = J$, then
it is the zero element of
$D_{\mathcal{Q}}^{s}(G^{L}_{1})$. If $I \neq J$,
 then from the fact $(I,J)$ and $(J,I)$ represent
 the
same element of
$\mathcal{Q}(\mathbb{Z} G^{L}_{1}, S_{2})$, which is the quotient $\mathbb{Z} G^{L}_{1}$ when considering
the action by the permutation group $S_{2}$,
we can assume that $I_{L}=e_{a,1}$ and
$J_{L} = e_{b,j(b)}$.

First consider the case $(I,J) \neq (e_{a,1}^{L}, e_{b,j(b)}^{L})$. This implies that
$(I_{1},J_{1}) \neq (e_{a,1}, e_{b,j(b)})$.
It follows from
Theorem \ref{2-12} that if $(I',J')$ is in
$G^{2L+1}_{1}$ with $t^{L+1}(I',J')= (I,J)$, then
the first $L$ entries of $I'$ and $J'$ are equal.
This means that
\[
(\gamma^{s})^{L+1}(I,J) = i^{L+1} \circ (t^{L+1})^{*}(I,J) = 0 \in \mathcal{Q}(\mathbb{Z} G^{L}_{1}, S_{2}).
\]
This in turn implies
$Q[(I,J),n] = 0$ in  $D_{\mathcal{Q}}^{s}(G^{L}_{1})$, for any positive integer $n$.

Finally, we consider $(I,J)= (e_{a,1}^{L}, e_{b,j(b)}^{L})$.
It follows that $i \circ t^{*}(I,J)$ is the sum
of $(e_{a,1}^{L}, e_{b,j(b)}^{L})$ and other terms, all of
the type considered above.This means that
$\gamma^{s}(e_{a,1}^{L}, e_{b,j(b)}^{L}) - (e_{a,1}^{L}, e_{b,j(b)}^{L})$ is zero in the limit under $\gamma^{s}$.
We conclude that
$Q[(e_{a,1}^{L}, e_{b,j(b)}^{L}), 1] = Q[(e_{a,1}^{L}, e_{b,j(b)}^{L}), n] \neq 0$,
for any positive integer $n$ and is a generator for
$D_{\mathcal{Q}}^{s}(G^{L}_{1})$.

For the other case, the group
$\mathcal{A}(\mathbb{Z} G^{L}_{1}, S_{2})$ is generated
by elements of the form $(I,J) - (J,I)$, where
$I \neq J$. It follows immediately from Theorem
\ref{2-12} that for any such $(I,J)$, if $(I',J')$
is in $G^{2L+1}_{1}$ and satisfies
$i^{L+1}(I',J')$, then the
last $L$ entries of $(I',J')$ are $(e_{a,1}^{L}, e_{b,j(b)}^{L})$ or $( e_{b,j(b)}^{L}, e_{a,1}^{L})$. This means that
$(\gamma^{u})^{L+1}(I,J) = t^{L+1} \circ (i^{L+1})^{*}(I,J)$ is a multiple of
$(e_{a,1}^{L}, e_{b,j(b)}^{L})$ or $( e_{b,j(b)}^{L}, e_{a,1}^{L})$. We also note that
if $(I,J) = (e_{a,1}^{L}, e_{b,j(b)}^{L})$ then the $(I',J')$
above is unique and we deduce
$\gamma^{u}(e_{a,1}^{L}, e_{b,j(b)}^{L}) = (e_{a,1}^{L}, e_{b,j(b)}^{L})$. The conclusion follows.
\end{proof}

Having identified the groups involved, our next task
is to compute the boundary maps between them. The first
step is the following technical lemma.

\begin{lemma}
\label{3-9}
The number  $K=1$  satisfies Lemma 2.7.1 of \cite{P1}
 for
 $ \overline{\rho}:(\Sigma_X,\sigma)\rightarrow (\overline{X},\overline{f})$.
\end{lemma}

\begin{proof}
Let $I,J, I', J'$ be in $\Sigma_{X}$
and satisfy $\overline{\rho}(I)=\overline{\rho}(J)$, $\overline{\rho}(I')=\overline{\rho}(J')$, $I_{n}=I'_{n}$,
 for all $n \geq 0$, and $J, J'$ are stably equivalent.
We need to prove that $J_{n}=J'_{n}$, for
all $n \geq 1$.

First observe that  $I_{n}=I'_{n}$, for all $n \geq 0$
implies that $\rho(I) = \rho(I')$ and hence we
have $\rho(J) = \rho(I)  = \rho(I') = \rho(J')$.

We proceed by considering the three cases given in the
first statement of Theorem \ref{2-12}. If $I=J$, then since
$I'$ and $J'$ are stably equivalent, we
must have $I'=J'$ also. If both $I,J$ and $I', J'$
satisfy the second condition, then $I=I'$ and
$J=J'$ and the conclusion holds.

 We consider the case that $I,J$
satisfies the third condition. Let $N$
 be the unique integer given
in the condition for $I,J$. First suppose that  $N \geq 0$. Then
$f^{N}(\rho(I)) =\rho(\sigma^{N}(I)) \in f^{-1}\{ p \} - \{ p \}$.
It follows then that $f^{N}(\rho(I'))  \in f^{-1}\{ p \} - \{ p \}$.
From this fact, it follows that the pair
$I', J'$ is also of the third type and the unique
$N'$ from that condition is equal to $N$.
We apply Lemma \ref{2-11} to the pairs $\sigma^{N}(I),
\sigma^{N}(J)$ and  $\sigma^{N}(I'),
\sigma^{N}(J')$.
Then $J_{n}=J'_{n}$ for all $n \geq N$
follows from the uniqueness statement in
the conclusion of Lemma \ref{2-11}.
On the other hand, for $ 0 \leq n < N$, we have
$J_{n} = I_{n} = I'_{n} = J'_{n}$ and we are done.

Next, we consider the case $N< 0$. Here we have
$J_{n} = e_{b, j(b)}$, for all $n \geq 1$.
 If $I',J'$ is also
of third type with $N' \geq 0$, the the same argument
above, reversing the roles of $I,J$ and $I',J'$
would imply $N = N' \geq  0$ which is a contradiction.
So either $I',J'$ is of the third type with
$N'< 0$ or it is of the second type. In
either case, we have $J'_{n} = e_{b,j(b)}$
for all $n \geq 1$
and we are done.
\end{proof}

We are now ready to compute the boundary map for
the first complex.

\begin{lemma}
\label{3-10}
Let $L$ be any integer such that
\[
\overline{\rho}: (\Sigma_{G_{L}}, \sigma)
\rightarrow (\overline{X}, \overline{f})
\]
 is  regular (see \cite{P1}).
Then we have
\[
d^{s,1}_{\mathcal{Q}}(\overline{\rho})_{1}(Q[(e_{a,1}^{L}, e_{b,j(b)}^{L}), 1])  = 2[w, 1].
\]
(Recall that $w=0$ if and only if $(X,f)$
is orientable.)
\end{lemma}

\begin{proof}
 In view of Lemma \ref{3-9}, we may
use $K=1$ to compute $\delta_{0}^{s,1}$ and $\delta_{1}^{s,1}$
from $D^{s}_{\mathcal{Q}}( G^{L}_{1}, S_{2})$ to
$D^{s}( G^{L+1})$. The former group is cyclic and is
generated by  $Q[(e_{a,1}^{L-1}, e_{b,j(b)}^{L-1}),1]$. Moreover, the
the map $i^{L+1}$ induces an isomorphism between
$D^{s}( G)$. It suffices to prove that
\[
i^{L}\circ \delta_{0}(e_{a,1}^{L-1}, e_{b,j(b)}^{L-1})
- i^{L}\circ \delta_{1}(e_{a,1}^{L-1}, e_{b,j(b)}^{L-1}) = 2 w.
\]
and we are done.

We must find all $I$ in $G^{L}$ such that there exists
$J$ with $(I,J)$ in $G^{L}_{1}$ and
$t(I,J) = (e_{a,1}^{L-1}, e_{b,j(b)}^{L-1})$. An almost
 complete
description is given by Lemma \ref{2-11}: this is
$I = ( I_{1}, \ldots, I_{L})$ with
 $I_{1}$ in $\mathcal{A}_{0}$ with
$f(I_{1} \cap U_{p}) \subset e_{a,1}$.
If $e_{i}$ is in $E_{a}$, then
both   $ ( e_{i,1}, e_{a,1}, \ldots, e_{a,1})$
and $ ( e_{i,j(i)}, e_{a,1}, \ldots, e_{a,1})$
appear in this sum. After applying $i^{L}$, we obtain
$2e_{i}$. If $e_{i}$ is in $E_{0}$, then one of these
two appears in the sum and the other does not.
After applying $i^{L}$ we get $e_{i}$. Finally, if
$e_{i}$ is in $E_{b}$, then neither appears in the list
and $e_{i}$ does not appear in the sum after applying $i^{L}$.

The reason this is not quite all, is that in Lemma
\ref{2-11}, it is possible that $I_{1} = e_{a,1}$
and $J_{1} = e_{b,j(b)}$. Let us assume
in such a case that $f(e_{i,j} \cap U_{i,j}) \subset e_{a,1}$
and $f(e_{i,j+1} \cap U_{i,j}) \subset e_{b,j(b)}$ for the other
case is similar.
Then we have $(e_{i,j}e_{a,1}^{L-1}, e_{i,j+1}e_{b,j(b)}^{L-1})$.
Applying $i^{L} \circ \delta_{0}$
 to such an element gives
 \[
 i^{L}(  e_{i,j}e_{a,1}^{L-1}) = i(e_{i,j}) = e_{i}.
 \]

Now let us compute
$i^{L}\circ \delta_{1}(e_{a,1}^{L-1}, e_{b,j(b)}^{L-1}) $.
First consider the extra elements we found at the end
of the last paragraph. Here we obtain the
element
\[
i^{L}(  e_{i,j+1}e_{a,1}^{L-1}) = i(e_{i,j+1}) = e_{i}.
 \]
 When we take the difference, these terms cancel.
 By an argument exactly analogous
 to the one above, what we are left with in our computation of
 $i^{L}\circ \delta_{1}(e_{a,1}^{L-1}, e_{b,j(b)}^{L-1}) $
 is that each element of $E_{a}$ does not appear, each
 element of $E_{0}$ appears with coefficient $1$ and each
 element of $E_{b}$ appears with coefficient $2$. Taking the difference we get
 \[
 2 Sum(E_{a}) - 2 Sum(E_{b}) = 2w,
 \]
 as claimed.
\end{proof}

We move on to the other boundary map.

\begin{lemma}
\label{3-11}
Let $L$ be any integer such that
\[
\overline{\rho}: (\Sigma_{G_{L}}, \sigma)
\rightarrow (\overline{X}, \overline{f})
\]
 is  regular (see \cite{P1}).
Then we have
\[
d^{u*,1}_{\mathcal{A}}(\overline{\rho})_{0}(a) = 2w^{*}(a)
 [(e_{a,1}^{L-1}, e_{b,j(b)}^{L-1}) - (e_{b,j(b)}^{L}, e_{a,1}^{L}), 1],
\]
for all $  a \in D^{u}(G)$.
\end{lemma}

\begin{proof}
We first consider the case that
 $e_{i}$ is in $E_{a}$. Recall that the canonical
isomorphism from $D^{s}(G)$ to $D^{s}(G^{L+1})$ is induced
by the map $i^{L*}$ which sends $e_{i}$ to the sum
of all paths $I=(I_{1}, \ldots, I_{L})$ with
$i(I_{1}) = e_{i}$.
To this element we wish to apply
$d^{u*,1}_{\mathcal{A}}(\overline{\rho})_{0}(a)$. Let $B$
denote the set of all pairs $(I,J)$ in
$G^{L}_{1}$ with $I$ fixed as above so that
\[
d^{u*,1}_{\mathcal{A}}(\overline{\rho})_{0}(a)
= Sum\{ t(I) \mid (I,J) \in B \} - Sum\{ t(J) \mid (I,J) \in B \} .
\]
Then we must sum over $I$, $i^{L}(I)=i(I_{0})=e_{i}$.
Divide $B$ into two subsets: those with $i(J_{0}) = i(I_{0})=e_{i}$
and $B_{0}$, the remaining ones.
The first group contributes nothing, since for each pair
$(I,J)$, both $I$ and $J$ appear when summing over
$i^{L}(I) = e_{i}$ and their contribution to the sum
is exactly opposite and cancel. We are left to compute
\[
Sum\{ t(I) \mid (I,J) \in B_{0} \} - Sum\{ t(J) \mid (I,J) \in B_{0} \}.
\]
It follows from Lemma \ref{3-9}
and the fact that we assume $e_{i}$ is in $E_{a}$ that $B_{0}$ is empty unless
$I = e_{i,1}e_{a,1}^{L-1}$ or $I = e_{i,j(i)}e_{a,1}^{L-1}$. In each
case, taking $t(I) -t(J)$ and
summing over $B_{0}$, we obtain $e_{a,1}^{L} - e_{b,j(b)}^{L}$.
Now summing over the two values of $I$, we get
$2( e_{a,1}^{L} - e_{b,j(b)}^{L} )$. We have verified the conclusion
for $e_{i}$ in $E_{a}$. The case that $e_{i}$ is
in $E_{b}$ is done in a similar way.

In the case that $e_{i}$ is in $E_{0}$, there are again
two $I$'s consider, but they are $I = e_{i,1}e_{a,1}^{L-1}$
and $I = e_{i,j(i)}e_{b,j(b)}^{L-1}$ (or reversing the first
entries). The terms $t(I) - t(J)$ are then opposite
for these two $I$'s and the total contribution is zero.
That is, we have shown the conclusion holds for
$e_{i}$ in $E_{0}$.
\end{proof}

We finally remark that the proof of Theorem \ref{3-1}
follows easily from Theorem \ref{3-7}, Lemmas \ref{3-8}
and \ref{3-10}. The proof of Theorem \ref{3-2}
follows easily from Theorem \ref{3-7}, Lemmas \ref{3-8}
and \ref{3-11}.

\end{document}